\documentclass[11pt,reqno]{amsart}
\usepackage{amsmath, amsthm, amscd, amsfonts, amssymb, graphicx, color}

\usepackage[colorlinks]{hyperref}

 \makeatletter
 \evensidemargin  \oddsidemargin
 \makeatother

\begin{document}
\thispagestyle{empty}
 \setcounter{page}{1}

\begin{center}
{\large\bf Convexity of parameter extensions of some relative operator entropies 
with a perspective approach\footnote{to appear in Glasgow Mathematical Journal}}  \vskip.25in

{\bf Ismail Nikoufar} \\
{\footnotesize \textit{Department of Mathematics, Payame Noor University, P.O. Box 19395-3697 Tehran, Iran}\\
[-1mm] \textit{e-mail:} {\tt nikoufar@pnu.ac.ir}}\\[2mm]

\end{center}
\vskip 5mm

\noindent{\footnotesize{\bf Abstract.}
In this paper,
we introduce two notions of a relative operator $(\alpha, \beta)$-entropy and a Tsallis relative operator $(\alpha, \beta)$-entropy
as two parameter extensions of the relative operator entropy and the Tsallis relative operator entropy.
We apply a perspective approach to prove the joint convexity or concavity of these new notions, under certain conditions concerning $\alpha$ and $\beta$.
Indeed, we give the parametric extensions, but in such a manner that they remain jointly convex or jointly concave.

\noindent{\footnotesize{\bf Significance Statement.}
What is novel here is that we convincingly demonstrate how our techniques can be used to give simple proofs for the old
and new theorems for the functions that are relevant to quantum statistics.
Our proof strategy shows that the joint convexity of the perspective of some functions plays a crucial role to give simple proofs for
the joint convexity (resp. concavity) of some relative operator entropies.
\\
\\
{\it Mathematics Subject Classification.} 81P45, 15A39, 47A63, 15A42, 81R15.\\
{\it Key words and phrases:} perspective function, generalized perspective function, relative operator entropy, Tsallis relative operator entropy.
\\
\\
{$\star$ The notions introduced here were used in our published paper \cite{Nikoufar-op}, when this paper was a draft}.

  \newtheorem{df}{Definition}[section]
  \newtheorem{rk}[df]{Remark}
   \newtheorem{lem}[df]{Lemma}
   \newtheorem{thm}[df]{Theorem}
   \newtheorem{pro}[df]{Proposition}
   \newtheorem{cor}[df]{Corollary}
   \newtheorem{ex}[df]{Example}

 \setcounter{section}{0}
 \numberwithin{equation}{section}

\vskip .2in

\begin{center}
\section{\bf Introduction}
\end{center}

Let $\mathcal H$ be an infinite-dimensional (separable) Hilbert space. Let
$B({\mathcal H})$ denote the set of all bounded linear operators on $\mathcal H$,
$B({\mathcal H})_{sa}$ the set of all self--adjoint operators, $B({\mathcal H})^+$ the set of all positive operators, and $B({\mathcal H})^{++}$
the set of all strictly positive operators.
A continuous real function
$f$ on $[0,\infty)$ is said to be operator monotone (more precisely, operator monotone increasing)
if $A\leq B$ implies $f(A)\leq f(B)$ for $A,B\in B({\mathcal H})_{sa}$.
For a self--adjoint operator $A$, the value $f(A)$ is defined via functional calculus
as usual.
The function $f$ is called operator convex if
\begin{equation}\label{convex-ineq1}
f(cA_1+(1-c)A_2)\leq cf(A_1)+(1-c)f(A_2)
\end{equation}
for all $A_1,A_2\in B({\mathcal H})_{sa}$ and $ c\in [0,1]$.
Moreover, the function $f$ is operator concave if
$-f$ is operator convex.
The function $g$ of two variables is called jointly convex if
\begin{equation}\label{convex-ineq2}
g(cA_1+(1-c)A_2,cB_1+(1-c)B_2)\leq cg(A_1,B_1)+(1-c)g(A_2,B_2)
\end{equation}
for all $A_1,A_2,B_1,B_2\in B({\mathcal H})_{sa}$ and $ c\in [0,1]$,
and jointly concave if the sign of inequality \eqref{convex-ineq2} is reversed.

Let $f$ and $h$ be two functions defined on $[0,\infty)$ and $(0,\infty)$, respectively and let $h$ be a strictly positive function, in the sense that,
$h(A)\in B({\mathcal H})^{++}$ for $A\in B({\mathcal H})^{++}$.
We introduced in \cite{Nikoufar-PNAS2011} a fully noncommutative perspective
of two variables (associated to $f$), by choosing an appropriate ordering, as follows:
$$
\Pi_{f}(A,B):=A^{1/2}f(A^{-1/2}BA^{-1/2})A^{1/2}
$$
for $A\in B({\mathcal H})^{++}$ and $B\in B({\mathcal H})_{sa}$.
We also introduced the operator version of a fully noncommutative generalized perspective of two variables
(associated to $f$ and $h$) as follows:
$$
\Pi_{f\Delta h}(A,B):=h(A)^{1/2}f(h(A)^{-1/2}Bh(A)^{-1/2})h(A)^{1/2}
$$
for $A\in B({\mathcal H})^{++}$ and $B\in B({\mathcal H})_{sa}$.
This beautiful contribution can surely affect quantum information theory and quantum statistical mechanics.
Noncommutative functional analysis
gives an appropriate framework for many of the calculations
in quantum information theory and nonclassical
techniques that clarify some of the conceptual problems
in operator convexity theory.
Note that the introduced perspective $\Pi_f$ with the operator monotone function $f$ is the operator mean introduced by Kubo and Ando in \cite{Kubo-Ando}.

By recalling that if for every continuous function $f$,
$f(A)$ commutes with every operator commuting with $A$ (including $A$ itself) and when we restricted to positive commuting matrices, i.e., $[A,B]=0$,
it becomes Effros's approach which is considered in \cite{Effros} as follows:
\begin{align*}
\Pi_{f}(A,B)&:=f(\frac{B}{A})A,\\
\Pi_{f\Delta h}(A,B)&:=f(\frac{B}{h(A)})h(A).
\end{align*}

Afterwards, we introduced the notion of the non-commutative perspective in \cite{Nikoufar-PNAS2011},
this notion was studied by Effros and Hansen in \cite{AFA}.
They proved that the non-commutative perspective of an operator
convex function is the unique extension of the corresponding commutative perspective
that preserves homogeneity and convexity.

In \cite{Nikoufar-PNAS2011}, we proved several striking matrix analogues
of a classical result for operator convex functions. Indeed, we proved the following two theorems
that entail the necessary and sufficient conditions for the joint convexity of
a fully noncommutative perspective and generalized perspective function.
We applied the affine version of Hansen--Pedersen--Jensen inequality \cite[Theorem 2.1]{Hansen} to prove the following result:

\begin{thm}\label{pers-cnx}
The function $f$ is operator convex (concave) if and only if the perspective function $\Pi_{f}$ is jointly convex (concave).
\end{thm}

We also used Hansen--Pedersen--Jensen inequality \cite[Theorem 2.1]{HansenPeder} to prove the following result:
%-----------------------------------------
%%Theorem
%-----------------------------------------
\begin{thm}\label{proc-nation}
Suppose that $f$ and $h$ are continuous functions with $f(0)<0$ and $h>0$. Then $f$ is operator convex and $h$ is operator concave if and only if the generalized perspective function $\Pi_{f\Delta h}$ is jointly convex.
\end{thm}

%-----------------------------------------
%%
%-----------------------------------------

In the `if' part of the above theorem, I would remark that we could allow $f(0)\leq 0$.
However, for the next applications in the `only if' part the condition $f(0)\neq 0$ is essential.
So, this theorem and its reverse one can be modified as follows.
Note that part (ii) of Theorem \ref{nikoufar-Ph.D.} is a correct version of \cite[Corollary 2.6 (i)]{Nikoufar-PNAS2011}
and parts (iii) and (v) are a complete and correct version of \cite[Corollary 2.6 (ii)]{Nikoufar-PNAS2011}.
We include the proofs for the convenience of the readers.

%-----------------------------------------
%%Theorem
%-----------------------------------------
\begin{thm} \label{nikoufar-Ph.D.}
Suppose that $f$ and $h$ are continuous functions and $h>0$.
\begin{itemize}
  \item[(i)]
If $f$ is operator convex and $h$ is operator concave with $f(0)\leq 0$, then the generalized perspective function
$\Pi_{f\Delta h}$ is jointly convex.
  \item[(ii)]
If $f$ and $h$ are operator concave with $f(0)\geq 0$, then the generalized perspective function
$\Pi_{f\Delta h}$ is jointly concave.
  \item[(iii)] If the generalized perspective function
$\Pi_{f\Delta h}$ is jointly convex (concave), then $f$ is operator convex (concave).
  \item[(iv)] If $f(0)>0$ and the generalized perspective function
$\Pi_{f\Delta h}$ is jointly convex (concave), then $h$ is operator convex (concave).
  \item[(v)] If $f(0)<0$ and the generalized perspective function
$\Pi_{f\Delta h}$ is jointly convex (concave), then $h$ is operator concave (convex).
 \end{itemize}
\end{thm}
\begin{proof}
(i)
For the strictly positive operators $A_1,A_2$, the self--adjoint operators $B_1,B_2$, and $ c\in [0,1]$
set $A:=cA_1+(1-c)A_2$ and $B:=cB_1+(1-c)B_2$.
Define $T_1:=(ch(A_1))^{1/2}h(A)^{-1/2}$ and $T_2:=((1-c)h(A_2))^{1/2}h(A)^{-1/2}$.
The concavity of $h$ gives $T_1^*T_1+T_2^*T_2\leq1$ and
the operator convexity of $f$ together with Hansen--Pedersen--Jensen inequality \cite{HansenPeder} imply
\begin{align*}
\Pi_{f\Delta h}(A,B)&=h(A)^{1/2}f(h(A)^{-1/2}B h(A)^{-1/2})h(A)^{1/2}\\
&=h(A)^{1/2}f\left(T_1^*h(A_1)^{-1/2}B_1h(A_1)^{-1/2}T_1+T_2^*h(A_2)^{-1/2}B_2h(A_2)^{-1/2}T_2\right)h(A)^{1/2}\\
&\leq h(A)^{1/2}\left(T_1^*f(h(A_1)^{-1/2}B_1h(A_1)^{-1/2})T_1\right. \\
&\ \ +\left. T_2^*f(h(A_2)^{-1/2}B_2h(A_2)^{-1/2})T_2\right)h(A)^{1/2}\\
&=c h(A_1)^{1/2}f(h(A_1)^{-1/2}B_1h(A_1)^{-1/2})h(A_1)^{1/2}\\
&\ \ +(1-c)h(A_2)^{1/2}f(h(A_2)^{-1/2}B_2h(A_2)^{-1/2})h(A_2)^{1/2}\\
&=c \Pi_{f\Delta h}(A_1,B_1)+(1-c)\Pi_{f\Delta h}(A_2,B_2).
\end{align*}
(ii) It follows from (i) by replacing $-f$ with $f$.

(iii) A simple computation shows that $f(A)=\frac{1}{h(1)}\Pi_{f\Delta h}(1,h(1)A)$.
Then, by using the joint convexity of $\Pi_{f\Delta h}$, for the self--adjoint operators $A_1,A_2$ and $0\leq c\leq 1$ we have
\begin{align*}
f(cA_1+(1-c)A_2)&=\frac{1}{h(1)}\Pi_{f\Delta h}(1,h(1)(cA_1+(1-c)A_2))\\
&=\frac{1}{h(1)}\Pi_{f\Delta h}(1,ch(1)A_1+(1-c)h(1)A_2)\\
&\leq\frac{1}{h(1)}(c\Pi_{f\Delta h}(1,h(1)A_1)+(1-c)\Pi_{f\Delta h}(1,h(1)A_2))\\
%&=\frac{1}{h(1)}(c h(1)f(A_1)+(1-c)h(1)f(A_2))\\
&=c f(A_1)+(1-c)f(A_2).
\end{align*}

(iv) It is obvious that $h(A)=\frac{1}{f(0)}\Pi_{f\Delta h}(A,0)$.
By using $f(0)>0$, for the strictly positive operators $A_1,A_2$ and $0\leq c\leq 1$ we get
\begin{align*}
h(cA_1+(1-c)A_2)&=\frac{1}{f(0)}\Pi_{f\Delta h}(cA_1+(1-c)A_2,0)\\
&\leq\frac{1}{f(0)}(c\Pi_{f\Delta h}(A_1,0)+(1-c)\Pi_{f\Delta h}(A_2,0))\\
&=ch(A_1)+(1-c)h(A_2).
\end{align*}

(v) The proof is similar to that of (iv).
\end{proof}
%-----------------------------------------
%%
%-----------------------------------------

%-----------------------------------------
%%section
%-----------------------------------------
\section{Parametric relative operator entropies}

Generalized entropies are used as alternate measures of an informational content.
Studies of generalized entropies allow to treat properties of the standard entropy in more general
setting. The connection between strong subadditivity of the von Neumann entropy and the Wigner--Yanase--Dyson
conjecture is a remarkable example (see \cite{Hiai, Jencova}).

In this section, we show usefulness of the notions of the perspective and the generalized perspective to obtain
the joint convexity of the (quantum) relative operator entropy, the joint concavity of the Fujii--Kamei relative operator entropy and
the Tsallis relative operator entropy, and moreover the joint convexity (concavity) of some other well-known operators.

Yanagi et al. \cite{Yanagi} defined the notion of the Tsallis relative operator entropy and gave its properties
and the generalized Shannon inequalities.
Furuichi et al. \cite{Furuichi} defined this notion as a parametric extension
of the relative operator entropy and proved some operator inequalities related to the Tsallis relative operator entropy.
For the strictly positive matrices $A,B$ and $0<\lambda\leq 1$,
$$
T_{\lambda}(A|B):=\frac{A^{\frac{1}{2}}(A^{-\frac{1}{2}}BA^{-\frac{1}{2}})^{\lambda}A^{\frac{1}{2}}-A}{\lambda}
$$
is called the Tsallis relative operator entropy between $A$ and $B$ \cite{Yanagi}.
We often rewrite the Tsallis relative operator entropy $T_{\lambda}(A|B)$ as
$$
T_{\lambda}(A|B)=A^{\frac{1}{2}}\ln_{\lambda}(A^{-\frac{1}{2}}BA^{-\frac{1}{2}})A^{\frac{1}{2}},
$$
where $\ln_{\lambda}X\equiv\frac{X^{\lambda}-1}{\lambda}$ for the positive operator $X$ \cite{Furuichi}.

We give a generalized notion of the Tsallis relative operator entropy and call it a Tsallis relative operator $(\alpha, \beta)$-entropy.
We define
$$
T_{\alpha, \beta}(A|B):=A^{\frac{\beta}{2}}\ln_{\alpha}(A^{-\frac{\beta}{2}}BA^{-\frac{\beta}{2}})A^{\frac{\beta}{2}}
$$
for the strictly positive operators $A,B$ and the real numbers $\alpha\neq 0, \beta$.
It is clear that every Tsallis relative operator $(\lambda, 1)$-entropy is the Tsallis relative operator entropy, i.e., $T_{\lambda}(A|B)=T_{\lambda, 1}(A|B)$.
We want to establish the joint convexity or concavity of $T_{\alpha, \beta}$ with a perspective approach, namely, we find the functions $f, h$ such that
$T_{\alpha, \beta}(A|B)=\Pi_{f\Delta h}(A,B)$. In particular, we reach a simple result on the joint convexity or concavity of the Tsallis relative operator entropy.

%-----------------------------------------
%%Lemma
%-----------------------------------------
\begin{lem}\label{ln}
The function $\ln_{\lambda}(t)$ is operator convex for $\lambda\in [1,2]$ and operator concave for $\lambda\in [-1,0)\cup(0,1]$.
\end{lem}
\begin{proof}
The result follows from the operator convexity or concavity of the elementary function $t^{\lambda}$.
\end{proof}

%-----------------------------------------
%%Theorem
%-----------------------------------------
\begin{thm}\label{gen.entropy}
The Tsallis relative operator $(\alpha, \beta)$-entropy is jointly convex for $\alpha\in [1,2]$ and $\beta\in [0,1]$.
 %and jointly concave for $\alpha\in [-1,0)\cup(0,1], \beta\in [0,1]$.
\end{thm}
\begin{proof}
Note that the Tsallis relative operator $(\alpha, \beta)$-entropy $T_{\alpha, \beta}$ is the generalized perspective of the functions $\ln_{\alpha}(t)$ and $t^{\beta}$,
in the sense that, $T_{\alpha, \beta}(A|B)=\Pi_{\ln_{\alpha}t\Delta t^{\beta}}(A,B)$.
Therefore, we obtain the result from Lemma \ref{ln} and Theorem \ref{nikoufar-Ph.D.} (i).
\end{proof}
We remark that the concavity assertion in Theorem \ref{gen.entropy} is doubtful.
In fact,  $\ln_{\alpha}(t)$ is operator concave for $\alpha\in[-1,0)\cup(0,1]$, where $\ln_{\alpha}(0)<0$, so Theorem \ref{nikoufar-Ph.D.} (ii) can not be applied.

Applying Theorem \ref{pers-cnx} the concavity assertion for the Tsallis relative operator entropy
is not doubtful.
%-----------------------------------------
%%cor
%-----------------------------------------
\begin{thm}\label{Tsallis-thm}
The Tsallis relative operator entropy is jointly convex for $\alpha\in [1,2]$ and jointly concave for $\alpha\in [-1,0)\cup(0,1]$.
\end{thm}
\begin{proof}
We have $\Pi_{\ln_{\alpha}t}(A,B)=T_{\alpha}(A|B)$
and the result follows from Lemma \ref{ln} and Theorem \ref{pers-cnx}.
\end{proof}

The notion of the relative operator entropy was introduced on strictly positive matrices in noncommutative information
theory by Fujii and Kamei \cite{Fujii-Kamei} as an extension of the operator entropy considered by
Nakamura and Umegaki \cite{Nakamura-Umegaki} and the relative operator entropy considered by
Umegaki \cite{Umegaki} as follows:
$$
S(A|B):=A^{\frac{1}{2}}(\log A^{-\frac{1}{2}}BA^{-\frac{1}{2}})A^{\frac{1}{2}}.
$$

Fujii et al. \cite{Fujii-Kamei} estimated the value of the relative operator entropy $S(A|B)$
by applying the Furuta's inequality and obtained the upper and lower bounds of $S(A|B)$.
It is obvious that $S(A|B)=\lim_{\alpha\to 0} T_{\alpha}(A|B)$ for $A,B>0$.
Hence, $S(A,B)$ is jointly concave by Theorem \ref{Tsallis-thm}.
We show that the joint concavity of the relative operator entropy is a simple consequence of the joint concavity of the perspective of the elementary function $f(t)=\log t$.
%%%%%%%%%%%%%%%%%%%%%%%%%%%%
%
%%%%%%%%%%%%%%%%%%%%%%%%%%%%
\begin{thm}\label{Fujii}
The Fujii--Kamei relative operator entropy $S(A|B)$ is jointly concave on the strictly positive operators $A,B$.
\end{thm}
\begin{proof}
The relative operator entropy $S(A|B)$ is the perspective of $\log t$ in the sense of our definition
and so Theorem \ref{pers-cnx} and the operator concavity of $\log t$ imply the result.
\end{proof}

Effros gave a new interesting proof for Lieb and Ruskai's result \cite{Lieb-Ruskai} (see Corollary 2.1 of \cite{Effros})
and now we provide simple proofs for the same results.
%%%%%%%%%%%%%%%%%%%%%%%%%%%%
%
%%%%%%%%%%%%%%%%%%%%%%%%%%%%
\begin{thm}
(i) The (quantum) relative entropy
$$
(\rho,\sigma) \mapsto H (\rho \| \sigma) = Trace\ \rho \log \rho - \rho \log \sigma
$$
is jointly convex on the commutative strictly positive operators $\rho,\sigma$.
\\
(ii) Part (i) holds for the noncommutative strictly positive operators $\rho,\sigma$.
\end{thm}
%%%%%%%%%%%%%%%%%%%%%%%%%%%%
%
%%%%%%%%%%%%%%%%%%%%%%%%%%%%
\begin{proof}
The following equalities show that parts (i) and (ii) are a simple application of Theorem \ref{Fujii}.
\\
(i) We have $H(\rho \| \sigma)=-Trace\ S(\rho |\sigma)$.
%and the result follows from  Theorem \ref{Fujii}.
\\
(ii) For the commuting operators $L_{\rho}$ and $R_{\sigma}$ we have
$$
\langle -S(L_{\rho}|R_{\sigma})(I),I\rangle=Trace\ \rho\log\rho-\rho\log\sigma,
$$
where $\langle\cdot,\cdot\rangle$ is the Hilbert-Schmidt inner product,
$L_{\rho}$ is the Left multiplication
by $\rho$ and $R_{\sigma}$ is the right multiplication by $\sigma$.
\end{proof}

%%%%%%%%%%%%%%%%%%%%%%%%%%%%
%
%%%%%%%%%%%%%%%%%%%%%%%%%%%%

Furuta \cite{Furuta} defined the generalized relative operator entropy for the strictly positive operators $A,B$ and $q\in\Bbb{R}$ by
$$
S_q(A|B)=A^{1/2}(A^{-1/2}BA^{-1/2})^q(\log A^{-1/2}BA^{-1/2})A^{1/2}.
$$
Using the notion of the generalized relative operator entropy, Furuta obtained the parametric extension of
the operator Shannon inequality and its reverse one.
Note that for $q=0$, we get the relative operator entropy between $A$ and $B$, i.e.,
$$
S_0(A|B)=S(A|B).
$$

A natural question now arises: What can we say about the joint convexity or concavity of the generalized relative operator entropy?
We will find a function $f$ such that $S_q(A|B)=\Pi_{f}(A,B)$.
We discuss this in the next section.

%-----------------------------------------
%%section
%-----------------------------------------
\section{Generalized transpose operator functions and its applications}

A motivation to write this section is to prove the joint convexity of the generalized relative operator entropy
introduced by Furuta \cite{Furuta}. We also give a parametric extension of this notion, namely, we introduce
the notion of a relative operator $(\alpha, \beta)$-entropy  and prove that $S_{\alpha, \beta}$ is jointly convex for $\alpha, \beta\in[0,1]$.
%By existing approaches in \cite{Nikoufar-PNAS2011} we can not answer this problem, so we first introduce the notion
%of the generalized dual function and then we investigate its operator convexity (concavity).

%-----------------------------------------added
%%def
%-----------------------------------------added
\begin{df}
Let $f$ and $h$ be continuous functions and $h>0$.
We define a generalized transpose function with respect to the functions $f$ and $h$ by
$$
f_{h}^{*}(t):=h(t)f(\frac{1}{h(t)}).
$$
In particular, the transpose function with respect to the function $f$ is defined by $f^{*}(t):=tf(t^{-1})$.
\end{df}

The following result is a straight forward consequence of Theorem \ref{pers-cnx}.
Indeed, we have $f^*(A)=Af(A^{-1})=A^{1/2}f(A^{-1/2}A^{-1/2})A^{1/2}=\Pi_{f}(A,1)$.
%-----------------------------------------
%%Theorem
%-----------------------------------------
\begin{thm}\label{thmf*}
Suppose that $f$ is a continuous function. Then,
$f$ is operator convex (concave) if and only if so is $f^{*}$.
\end{thm}

The trace operation plays a central role in quantum statistical mechanics.
The mapping $A\mapsto Trace Kf(A)$ is certainly
convex when $K>0$ and  $f$ is operator convex.

%%%%%%%%%%%%%%%%%%%%%%%%%%%%
%
%%%%%%%%%%%%%%%%%%%%%%%%%%%%
\begin{cor}
The von Neumann entropy $S(\rho)=-Trace \rho\log \rho$ is operator concave on the strictly positive operator $\rho$.
\end{cor}
\begin{proof}
For the operator concave function $f(t)=\log t$ we have $f^{*}(t)=-t\log t$.
Using Theorem \ref{thmf*}, we deduce the function $-t\log t$ is operator concave and hence we obtain the desired result.
\end{proof}

%-----------------------------------------
%%Theorem
%-----------------------------------------
\begin{thm}\label{thm1}
Suppose that $f$ and $h$ are continuous functions and $h>0$.
\begin{itemize}
  \item[(i)]
If $f$ is operator convex with $f(0)\leq 0$ and $h$ is operator concave,
then $f^{*}_{h}$ is operator convex.
  \item[(ii)]
If $f$ and $h$ are operator concave with $f(0)\geq 0$, then $f^{*}_{h}$ is operator concave.
 \end{itemize}
\end{thm}
%-----------------------------------------
%%proof
%----------------------------------------
\begin{proof}
(i) Let $f$ be operator convex and $h$ operator concave. Then, it follows from Theorem \ref{nikoufar-Ph.D.} (i) that $\Pi_{f\Delta h}$ is jointly convex.
The fact that a jointly convex function is convex in each of its arguments separately and
$\Pi_{f\Delta h}(A,1)=f^{*}_{h}(A)$ imply  $f^{*}_{h}$ is operator convex.

(ii) The result comes from Theorem \ref{nikoufar-Ph.D.} (ii).
\end{proof}

The proof of the following lemma is straightforward.
%-----------------------------------------
%%Lemma
%-----------------------------------------
%%%%%%%%%%%%%%%%%%%%%%%%%%%%
%
%%%%%%%%%%%%%%%%%%%%%%%%%%%%
\begin{lem}
The function $f$ is operator convex (concave) if and only if so is $f_{\varepsilon}$ for every $\varepsilon>0$, where
$f_{\varepsilon}(t):=f(t+\varepsilon)$.
\end{lem}

%%%%%%%%%%%%%%%%%%%%%%%%%%%%%%%%%%%%%%%%%%%%%%%%%%%%%%%%%%%%%%%%%%%%%%%%%%%%%%%%%%%%%%%%%%%%%%%%%%%%%%%%%%%%%%
%%%%%%%%%%%%%%%%%%%%%%%%%%%%%%%%%%%%%%%%%%%%%%%%%%%%%%%%%%%%%%%%%%%%%%%%%%%%%%%%%%%%%%%%%%%%%%%%%%%%%%%%%%%%%%

Let $f$ be a twice differentiable function on $[0,\infty)$.
Define $k(t):=t^qf(t)$ for $0\leq q\leq 1$ and consider
$$
I_q:=\{t\geq 0: k''(t)\geq 0\}.
$$
Clearly, $k$ is not convex on $\Bbb{R}^+-I_q$ and hence is not operator convex on outside of $I_q$.
We show that under some assumptions $k$ is operator convex on $I_q$.

%TTTTTTTTTTTTTTTTTTTTTTTTTTTT
%
%
%TTTTTTTTTTTTTTTTTTTTTTTTTTTT  a twice differentiable function and
\begin{lem}\label{tqlogt01}
If $f$ is operator monotone on $[0,\infty)$ such that $f(0)\leq 0$ and $\lim_{t\to\infty}\frac{f(t)}{t}=0$,
then the function $k$ is operator convex on $I_q$.
\end{lem}
%PPPPPPPPPPPPPPPPPPPPPPPPPPPP
%
%
%PPPPPPPPPPPPPPPPPPPPPPPPPPPP
\begin{proof}
The operator monotone function $f$ on $[0,\infty)$ can be represented as
$$
f(t)=f(0)+\beta t+\int_{0}^{\infty}\frac{\lambda t}{\lambda +t}d\mu(\lambda),
$$
where $\beta\geq 0$ and $\mu$ is a positive measure on $[0,\infty)$; see \cite[Chapter V]{Bhatia}.
Since $\lim_{t\to\infty}\frac{f(t)}{t}=0$, $\beta=0$. So by multiplying both sides to $t^q$ we have
$$
t^qf(t)=f(0)t^q+\int_{0}^{\infty}\frac{\lambda t^{1+q}}{\lambda +t}d\mu(\lambda).
$$
The function $f(0)t^q$ is operator convex.
Indeed, it is sufficient to prove that the function $\frac{\lambda t^{1+q}}{\lambda +t}$ is operator convex. Define $g(t)=t^q$
and consider two cases:
(i) For $\lambda>1$ the function $h(t)=\frac{t+1/\lambda}{t}g(t)$
is operator monotone by \cite[Corollary V.3.12]{Bhatia}.
So \cite[Theorem 2.4]{HansenPeder} and \cite[Problem V.5.7]{Bhatia} show that the function $th(t^{-1})^{-1}$ is operator convex.
(ii) For $0<\lambda<1$ the function $h_1(t)=\frac{t+\lambda}{t}g(t)$ is operator monotone by \cite[Corollary V.3.12]{Bhatia}.
So \cite[Problem V.5.7]{Bhatia} entails that the function $s(t)=h_1(\lambda t^{-1})^{-1}$ is operator monotone.
This implies the function $s(\frac{t}{\lambda})$ is also operator monotone.
Therefore, the function $\lambda^{2q}ts(\frac{t}{\lambda})$ is operator convex by \cite[Theorem 2.4]{HansenPeder}.
In each of the cases a simple calculation shows that the function $th(t^{-1})^{-1}$ in the case (i) and
the function $\lambda^{2q}ts(\frac{t}{\lambda})$ in the case (ii) are equal to the function $\frac{\lambda t^{1+q}}{\lambda +t}$.
\end{proof}

%We employ \cite[Theorem 2.5]{Nikoufar-PNAS2011} to derive the operator convexity of the function $f(t)=t^q\log t$ on strictly positive operators $A$.
%TTTTTTTTTTTTTTTTTTTTTTTTTTTT
%
%
%TTTTTTTTTTTTTTTTTTTTTTTTTTTT

%We are going to show that the function $f(t)=t^q\log t$ is operator convex on $[0,\infty)$ for $0<\delta\leq q\leq 1$.
%For other applications of the following two lemmas, we refer the readers to \cite{Nikoufar-mia}.

\begin{lem}\label{tqlogt12}
The function $k(t)=t^q\log t$ is operator convex on $J_q:=[0,e^{\frac{2q-1}{q(1-q)}}]$ for $0\leq q\leq 1$.
\end{lem}
%PPPPPPPPPPPPPPPPPPPPPPPPPPPP
%
%
%PPPPPPPPPPPPPPPPPPPPPPPPPPPP
\begin{proof}
Let $\varepsilon\in(0,1)$. Then, the function $f_{\varepsilon}(t)=\log(t+\varepsilon)$ satisfies in the assumptions of Lemma \ref{tqlogt01} and so
$k_{\varepsilon}(t)=t^q\log(t+\varepsilon)$ is operator convex on an interval $J_{q,\varepsilon}\subseteq J_q$.
Hence, when $\varepsilon\to 0$ we see that the function $k$ is operator convex on $J_q$.
%Let $s\in [0,1]$. Then, the function $t^{1-s}\log t^s$ is operator monotone and so is $st^{1-s}\log t$.
%By applying \cite[Theorem 2.4]{HansenPeder} we see that the function $st^{2-s}\log t$ is operator convex and hence is $t^{2-s}\log t$.
%Therefore, $t^q\log t$ is operator convex for $1\leq q\leq 2$.
\end{proof}

We now prove the main result of this section and its generalization.
%-----------------------------------------
%%Theorem
%-----------------------------------------
\begin{cor}
The generalized relative operator entropy $S_q(A|B)$ is jointly convex on the strictly positive operators $A,B$ with spectra in $J_q$ and $0\leq q\leq 1$.
\end{cor}
\begin{proof}
The generalized relative operator entropy $S_q(A|B)$ is the perspective of the operator convex function $t^q\log t$, $t>0$ and so Lemma \ref{tqlogt12} and Theorem \ref{pers-cnx}
ensure that $S_q(A|B)$ is jointly convex.
\end{proof}

We introduce a relative operator $(\alpha,\beta)$-entropy (two parameters relative operator entropy) as follows:
$$
S_{\alpha,\beta}(A|B)=A^{\frac{\beta}{2}}(A^{-\frac{\beta}{2}}BA^{-\frac{\beta}{2}})^{\alpha}(\log A^{-\frac{\beta}{2}}BA^{-\frac{\beta}{2}})A^{\frac{\beta}{2}}
$$
for the strictly positive operators $A,B$ and the real numbers $\alpha, \beta$.
We consider its convexity or concavity properties.
In particular, we have
$S_{q,1}(A|B)=S_{q}(A|B)$ and $S_{0,1}(A|B)=S(A|B)$.

%-----------------------------------------
%%Theorem
%-----------------------------------------
\begin{thm}\label{S-alpha}
The relative operator $(\alpha,\beta)$-entropy $S_{\alpha,\beta}(A|B)$ is jointly convex
on the strictly positive operators $A,B$ with spectra in $J_{\alpha}$ and $0\leq \alpha, \beta\leq 1$.
\end{thm}
\begin{proof}
Consider $f(t):=t^{\alpha}\log t$ and $h(t):=t^{\beta}$. Then, $f(0)=0$ and $S_{\alpha,\beta}(A|B)=\Pi_{f\Delta h}(A,B)$.
Using Theorem \ref{nikoufar-Ph.D.} (i) and Lemma \ref{tqlogt12} we deduce the generalized perspective of the operator convex function $f$ and
the operator concave function $h$ is jointly convex so that the relative operator $(\alpha,\beta)$-entropy $S_{\alpha,\beta}(A|B)$ is jointly convex.
\end{proof}

{\small
%----------------------------------------------------------------------%

}
\end{document}